\documentclass[a4paper,leqno,10pt]{amsart}

\usepackage{epsf,graphicx,epsfig}
\usepackage{amscd}
\usepackage{amsmath,latexsym,amssymb,amsthm}
\usepackage[nospace,noadjust]{cite}
\usepackage{textcomp}
\usepackage{dsfont}
\usepackage{setspace,cite}
\usepackage{lscape,fancyhdr,fancybox}
\usepackage{stmaryrd}
\usepackage[all,cmtip]{xy}
\usepackage{tikz}
\usepackage{cancel}
\usetikzlibrary{shapes,arrows,decorations.markings}
\usepackage{graphicx}

\usepackage{color}
\usepackage{url}
\usepackage{enumerate}
\usepackage[mathscr]{euscript}

  \theoremstyle{plain}

\swapnumbers
    \newtheorem{thm}{Theorem}[section]
    \newtheorem{prop}[thm]{Proposition}
     
   \newtheorem{lemma}[thm]{Lemma}
    \newtheorem{corollary}[thm]{Corollary}
    
    \newtheorem{subsec}[thm]{}
\theoremstyle{definition}
    \newtheorem{defn}[thm]{Definition}

\theoremstyle{remark}

\title{}
\author{}
\date{}
\usepackage{amssymb}

\usepackage{hyperref}
\hypersetup{
	colorlinks,
	citecolor=blue,
	filecolor=black,
	linkcolor=blue,
	urlcolor=black
}

\begin{document}

\title{Rota-Baxter operators on involutive associative algebras}

\author{Apurba Das}
\address{Department of Mathematics and Statistics,
Indian Institute of Technology, Kanpur 208016, Uttar Pradesh, India.}
\email{apurbadas348@gmail.com}

\curraddr{}
\email{}






\begin{abstract}
In this paper, we consider Rota-Baxter operators on involutive associative algebras. We define cohomology for Rota-Baxter operators on involutive algebras that governs the formal deformation of the operator. This cohomology can be seen as the Hochschild cohomology of a certain involutive associative algebra with coefficients in a suitable involutive bimodule. We also relate this cohomology with the cohomology of involutive dendriform algebras. Finally, we show that the standard Fard-Guo construction of the functor from the category of dendriform algebras to the category of Rota-Baxter algebras restricts to the involutive case.\\
\end{abstract}

\noindent

\thispagestyle{empty}

\maketitle

\tableofcontents

\vspace{0.2cm}

\section{Introduction}
Rota-Baxter operators are an algebraic abstraction of the integral operator that was first introduced by Baxter in his study of the fluctuation theory in probability \cite{baxter}. The study of Rota-Baxter operators was further developed by Rota \cite{rota} and Cartier \cite{cart} in relationship with combinatorics. They were found important applications in the Connes-Kreimer's algebraic approach of the renormalization of quantum field theory \cite{conn}. Rota-Baxter operators are also useful to study splitting of algebras. Namely, Rota-Baxter operators give rise to dendriform algebras which are splitting of associative algebras \cite{loday,aguiar}. In \cite{fard-guo} Ebrahimi-Fard and Guo constructs the universal enveloping Rota-Baxter algebra of a dendriform algebra in view of the standard universal enveloping algebra of a Lie algebra. The cohomology and deformation problem of associative Rota-Baxter operators (more generally of relative Rota-Baxter operators \cite{uchino}) has been recently studied by the author in \cite{das-rota}.

\medskip

On the other hand, classical algebras such as associative algebras, $A_\infty$-algebras and $L_\infty$-algebras equipped with involutions are studied in the last few years. An involutive associative algebra is an associative algebra $A$ together with a linear map $*: A \rightarrow A, ~ a \mapsto a^*$ satisfying $a^{**} = a$ and $(ab)^* = b^*a^*$, for $a,b \in A$. Such involutive algebras first appeared in mathematical physics in the context of an unoriented version of topological field theory \cite{cos}. Involutive algebras often appear in the standard constructions of algebras arising in geometric contexts, when the underlying geometric object has an involution \cite{braun,cos}. For example, the de Rham cohomology of a manifold with an involution carries an involutive $A_\infty$-algebra structure \cite{loday-val}. In \cite{braun} Braun has defined Hochschild cohomology of involutive associative algebras. An interpretation of Braun's Hochschild cohomology is given by the authors in \cite{fer-gian} using involutive Bar complex which led them to also introduce Hochschild homology of involutive associative algebras. Recently, with Saha, the present author gave a more explicit description of Hochschild cohomology of involutive associative algebras \cite{das-saha}. More precisely, they defined involutive dendriform algebras, their cohomology and find relations with the Hochschild cohomology of involutive associative algebras.

\medskip

Our aim in this paper is to study (relative) Rota-Baxter operators on involutive associative algebras. Let $(A, *)$ be an involutive associative algebra and $(M, *)$ be an involutive $A$-bimodule. A linear map $T: M \rightarrow A$ is said to be a relative Rota-Baxter operator on $A$ with respect to the involutive $A$-bimodule $M$ if $T$ satisfies $T(u^*) = T(u)^*$ and the following identity
\begin{align*}
T(u) T(v) = T (u T(v) + T(u) v ),~ \text{ for } u, v \in M.
\end{align*}
From the last identity, it follows that $T$ is a relative Rota-Baxter operator on the ordinary associative algebra $A$ with respect to the ordinary $A$-bimodule $M$. Here the word `ordinary' means that we are not considering the involution.
By definition, a Rota-Baxter operator on an involutive associative algebra $A$ is a relative Rota-Baxter operator on the involutive algebra $A$ with respect to itself. A (relative) Rota-Baxter operator on an involutive algebra induces an involutive dendriform algebra structure on the domain of the operator. 
Using Gerstenhaber's bracket on involutive Hochschild cochains and Voronov's derived bracket \cite{voro}, in Section \ref{sec-mc}, we construct a graded Lie algebra whose Maurer-Cartan elements are relative Rota-Baxter operators. Thus, a relative Rota-Baxter operator $T$ on an involutive algebra $A$ with respect to an involutive $A$-bimodule $M$ induces cohomology, called the cohomology of $T$. 

\medskip

In Section \ref{sec-coho}, we show that the cohomology of $T$ introduced in the previous section can be seen as the Hochschild cohomology of an involutive associative algebra with coefficients in a suitable involutive bimodule. For a relative Rota-Baxter operator $T$ on an involutive associative algebra $A$ with respect to an involutive bimodule $M$, we show that the ordinary cohomology of $T$ (viewed as a relative Rota-Baxter operator on the ordinary algebra $A$ with respect to the ordinary bimodule $M$) has a direct sum decomposition of the involutive cohomology of $T$ and a skew-factor. Finally, we obtain a morphism from the cohomology of a relative Rota-Baxter operator $T$ and the cohomology of the induced involutive dendriform algebra.

\medskip

The classical deformation theory of Gerstenhaber \cite{gers} has been extended to associative Rota-Baxter operators in \cite{das-rota}. In Section \ref{sec-def}, we study deformations of a relative Rota-Baxter operator $T$ on an involutive associative algebra with respect to an involutive bimodule. Our main results in this section are similar to the results of \cite{das-rota}. We show that the linear term in a formal deformation of $T$ is a $1$-cocycle in the cohomology of $T$, called the infinitesimal of the deformation. Moreover, equivalent deformations have cohomologous infinitesimals. Given a finite order deformation of $T$, we associate a $2$-cocycle in the cohomology complex of $T$, called the obstruction $2$-cocycle. When the corresponding cohomology class vanishes, the given deformation extends to deformation of next order.

\medskip

Finally, in Section \ref{sec-fard-guo}, we first recall the construction of the universal enveloping Rota-Baxter algebra of a dendriform algebra. Then we show that this construction restricts to the corresponding algebras equipped with involutions.

\medskip

All vector spaces, linear maps and tensor products are over a field $\mathbb{K}$ of characteristic $0$.

\section{(Relative) Rota-Baxter operators on involutive associative algebras}\label{sec-mc}
In this section, we introduce relative Rota-Baxter operators on involutive associative algebra with respect to an involutive bimodule. A particular case is given by Rota-Baxter operators on involutive algebra. We construct a graded Lie algebra whose Maurer-Cartan elements are relative Rota-Baxter operators.

\subsection{Involutive associative algebras and Hochschild cohomology}
An involution on a vector space $V$ is a linear map $* : V \rightarrow V,$ $v \mapsto v^*$ satisfying $v^{**} = v$, for all $v \in V$. Thus, an involution on $V$ is an invertible linear map on $V$ that equals to its inverse.

\begin{defn}
An involutive associative algebra is an associative algebra $A$ together with an involution $* : A \rightarrow A$ that satisfies $(ab)^* = b^* a^*$, for all $a, b \in A$.
\end{defn}

A morphism between involutive associative algebras is a morphism between underlying algebras preserving the involutions.
Let $A$ be an involutive associative algebra. An involutive $A$-bimodule is an ordinary $A$-bimodule $M$ together with an involution $* : M \rightarrow M$ that satisfies $(au)^* = u^* a^*$ and $(ua)^* = a^* u^*,$ for $a \in A, u \in M$. 

In this case, the direct sum $A \oplus M$ carries an involutive associative algebra structure (called the semi-direct product) with the involution $(a,u)^* = (a^*, u^*)$ and the product
\begin{align*}
(a,u) \cdot (b, v) = (ab, av + ub).
\end{align*}

In the following, we recall the Hochschild cohomology of an involutive associative algebra $A$ with coefficients in an involutive $A$-bimodule $M$.  First consider the ordinary Hochschild cochain complex $\{ C^\bullet_{\mathrm{Hoch}} (A, M), \delta_{\mathrm{Hoch}} \},$ where $C^n_{\mathrm{Hoch}}(A, M) = \mathrm{Hom}(A^{\otimes n}, M)$ for $  n \geq 0$ and the differential $\delta_{\mathrm{Hoch}} : C^n_{\mathrm{Hoch}}(A, M) \rightarrow C^{n+1}_{\mathrm{Hoch}}(A, M)$ given by
\begin{align*}
( \delta_{\mathrm{Hoch}} f ) (a_1, \ldots, a_{n+1}) =~& a_1 f(a_2, \ldots, a_{n+1}) 
+ \sum_{i=1}^n (-1)^i f (a_1, \ldots, a_{i-1}, a_i a_{i+1}, \ldots, a_{n+1}) \\
~&+ (-1)^{n+1} f (a_1, \ldots, a_n) a_{n+1}.
\end{align*}

For $n \geq 0$, consider the collection of subspaces $i C^n_{\mathrm{Hoch}}(A, M) \subset C^n_{\mathrm{Hoch}}(A, M)$ given by $i C^0_{\mathrm{Hoch}}(A, M) = \{ m \in C^0_{\mathrm{Hoch}}(A, M) = M | m^* = - m \} $ and $$i C^n_{\mathrm{Hoch}}(A, M) = \{ f \in C^n_{\mathrm{Hoch}}(A, M) |~ f (a_1, \ldots, a_n )^* = (-1)^{\frac{(n-1)(n-2)}{2}} f(a_n^*, \ldots, a_1^*) \},~ \text{ for } n \geq 1.$$ It has been shown in \cite{das-saha} that $\{ i C^\bullet_\mathrm{Hoch} (A, M), \delta_{\mathrm{Hoch}} \} $ is a subcomplex of the ordinary Hochschild complex and the cohomology of this subcomplex is called the Hochschild cohomology of the involutive algebra $A$ with coefficients in the involutive bimodule $M$. 

Next we show that the classical Gerstenhaber bracket on ordinary Hochschild cochains passes onto the involutive Hochschild cochains. Let us first recall the classical Gerstenhaber bracket \cite{gers-ring}. For $f \in C^m_{\mathrm{Hoch}}(A, A)$ and $g \in C^n_{\mathrm{Hoch}}(A, A)$, the Gerstenhaber bracket $[f,g] \in C^{m+n-1}_{\mathrm{Hoch}}(A,A)$ is given by 
\begin{align}\label{gers-brkt}
&[f,g] = \sum_{i=1}^m (1-)^{(i-1) (n-1)} f \circ_i g - (-1)^{(m-1)(n-1)} \sum_{i=1}^n (-1)^{(i-1)(m-1)} g \circ_i f,~ \text{ where }\\
&( f \circ_i  g)(a_1, \ldots, a_{m+n-1}) = f (a_1, \ldots, a_{i-1}, g( a_i, \ldots, a_{i+n-1}), a_{i+n}, \ldots, a_{m+n-1}). \nonumber
\end{align}

With this notation, we have the following.
\begin{prop}
If $f \in iC^m_{\mathrm{Hoch}}(A, A)$ and $g \in iC^n_{\mathrm{Hoch}}(A, A)$, then $[f,g] \in iC^{m+n-1}_{\mathrm{Hoch}}(A, A).$
\end{prop}
\begin{proof}
First observe that
\begin{align*}
(f \circ_i  g)(a_1, \ldots, a_{m+n-1})^* = (-1)^{\frac{(m-1)(m-2)+(n-1)(n-2)}{2}} ~(f \circ_{m-i+1} g )(a_{m+n-1}^*, \ldots, a_1^*).
\end{align*}
Hence
\begin{align*}
&(\sum_{i=1}^m (-1)^{(i-1)(n-1)} f \circ_i g ) (a_1, \ldots, a_{m+n-1})^* \\
&= (-1)^{\frac{(m-1)(m-2)+(n-1)(n-2)}{2}} \sum_{i=1}^m (-1)^{(i-1)(n-1)} ~( f \circ_{m-i+1} g ) (a_{m+n-1}^*, \ldots, a_1^*) \\
&= (-1)^{\frac{(m-1)(m-2)+(n-1)(n-2)}{2} + (m-1)(n-1)}  \sum_{i=1}^m (-1)^{(m-i)(n-1)} ~( f \circ_{m-i+1} g ) (a_{m+n-1}^*, \ldots, a_1^*).
\end{align*}
Therefore,
\begin{align*}
[f,g](a_1, \ldots, a_{m+n-1})^* =~& (-1)^{\frac{(m-1)(m-2)+(n-1)(n-2)}{2} + (m-1)(n-1)} ~[f,g] (a_{m+n-1}^*, \ldots, a_1^*) \\
=~& (-1)^{\frac{(m+n-2)(m+n-3)}{2}} ~[f,g] (a_{m+n-1}^*, \ldots, a_1^*).
\end{align*}
This shows that $[f,g] \in iC^{m+n-1}_{\mathrm{Hoch}}(A, A).$
\end{proof}

\subsection{Relative Rota-Baxter operators}

\begin{defn}
Let $A$ be an associative algebra. A linear map $R : A \rightarrow A$ is a Rota-Baxter operator on $A$ if $R$ satisfies
\begin{align}\label{rota-iden}
R(a) R(b) = R (a R(b) + R(a) b ), ~ \text{ for } a, b \in A.
\end{align}
\end{defn}
If $A$ is an involutive associative algebra, then a linear map $R : A \rightarrow A$ is said to be a Rota-Baxter operator on $A$ if $R(a^*) = R(a)^*$ and satisfies (\ref{rota-iden}).

\begin{defn}
Let $A$ be an involutive associative algebra and $M$ be an involutive $A$-bimodule. A linear map $T: M \rightarrow A$ is called a relative Rota-Baxter operator on $A$ with respect to the involutive $A$-bimodule $M$ if $T$ satisfies $T(u^*) = T(u)^*$ and
\begin{align*}
T(u) T(v) = T ( u T(v) + T(u) v),~ \text{ for } u,v \in M.
\end{align*}
\end{defn}
They are also called involutive relative Rota-Baxter operators. Thus, it follows that a Rota-Baxter operator on an involutive associative algebra $A$ is a relative Rota-Baxter operator on $A$ with respect to the involutive bimodule $A$ itself.

\begin{prop}\label{dirac-char}
Let $A$ be an involutive associative algebra and  $M$ be an involutive $A$-bimodule. A linear map $T: M \rightarrow A$ is a relative Rota-Baxter operator on $A$ with respect to the bimodule $M$ if and only if the graph of $T$,
\begin{align*}
\mathrm{Gr}(T) =\{ (Tu,u)| ~u \in M \}
\end{align*}
is an involutive subalgebra of the semi-direct product $A \oplus M$. 
\end{prop}



Let $T$ (resp. $T'$) be a relative Rota-Baxter operator on an involutive associative algebra $A$ with respect to an involutive $A$-bimodule $M$ (resp. on an involutive associative algebra $A'$ with respect to an involutive $A'$-bimodule $M'$).

\begin{defn}\label{mor-defn}
A morphism from $T$ to $T'$ consists of a pair $(\phi, \psi)$ in which $\phi : A \rightarrow A'$ is an involutive algebra morphism and $\psi : M \rightarrow M'$ is a linear map satisfying $\psi (u^*) = \psi (u)^*$ and
\begin{align*}
T' \circ \psi = \phi \circ T, \quad \psi ( au) = \phi (a) \psi (u) ~~~ \text{ and } ~~~ \psi (u a) = \psi (u) \phi (a),
\end{align*} 
for all $a \in A$ and $u \in M$. A morphism $(\phi, \psi)$ is called an isomorphism if $\phi$ and $\psi$ are both linear isomorphisms.
\end{defn}

In \cite{aguiar} Aguiar showed that a (relative) Rota-Baxter operator induces a dendriform structure. Here we observe the corresponding result in the involutive case.

\begin{defn}
A dendriform algebra is a vector space $D$ together with bilinear operations $\prec, \succ : D \otimes D \rightarrow D$ satisfying the following three identities
\begin{align*}
( a \prec b ) \prec c = a \prec ( b \prec c + b \succ c ), ~~~
( a \succ b ) \prec c = a \succ ( b \prec c), ~~~
( a \prec b + a \succ b ) \succ c = a \succ ( b \succ c),
\end{align*}
for all $a , b, c \in D$. A dendriform algebra as above may be denoted by the triple $(D, \prec, \succ)$.
\end{defn}
An involutive dendriform algebra is a dendriform algebra $(D, \prec, \succ)$ together with an involution $* : D \rightarrow D$ that satisfies $(a \prec b)^* = b^* \succ a^*$ (equivalently, $(a \succ b)^* = b^* \prec a^*$), for all $a, b \in D$.

\begin{prop}\label{inv-rota-inv-dend}
Let $T$ be a relative Rota-Baxter operator on an involutive associative algebra $A$ with respect to an involutive $A$-module $M$. Then $M$ carries an involutive dendriform algebra structure with products
\begin{align*}
u \prec v = u T(v) \quad \text{ and } \quad u \succ v = T(u)v, ~ \text{ for } u, v \in M.
\end{align*}
\end{prop}

\subsection{Gauge transformations}
Let $A$ be an involutive associative algebra and $M$ be an involutive $A$-bimodule. Let $T: M \rightarrow A$ be a relative Rota-Baxter operator. Consider the involutive subalgebra $\mathrm{Gr}(T) \subset A \oplus M$
of the semi-direct product. 

For any involutive $1$-cochain $B \in i C^1_{\mathrm{Hoch}}(A, M)$, we consider the deformed subspace
\begin{align*}
\tau_B ( \mathrm{Gr}(T) ) = \{ (Tu , u + B (Tu) )|~ u \in M \} \subset A \oplus M.
\end{align*}

\begin{lemma}
If $B \in i C^1_{\mathrm{Hoch}}(A, M)$ is an involutive Hochschild $1$-cocycle then
the subspace $\tau_B ( \mathrm{Gr}(T) ) \subset A \oplus M$ is an involutive subalgebra of the semi-direct product $A \oplus M$.
\end{lemma}

\begin{proof}
For any $u, v \in M$, we have
\begin{align*}
&(Tu , u + B (Tu) ) \cdot (Tv , v + B (Tv) ) \\
&= \big( T(u) T(v),~ T(u) v + uT(v) + T(u) (B(Tv)) + (B(Tu)) T(v) \big)  \\
&= \big( T(u)T(v) , ~ T(u) v + uT(v) + B (T(u)T(v)) \big) \quad (\text{since } B \text{ is a } 1\text{-cocycle}). 
\end{align*}
This is in $\tau_B ( \mathrm{Gr}(T) )$ as $T$ is a relative Rota-Baxter operator. Finally, this is an involutive subspace as $B$ is an involutive $1$-cochain.
\end{proof}

We now ask the question whether the involutive subalgebra $\tau_B ( \mathrm{Gr}(T) )$ is the graph of a new involutive relative Rota-Baxter operator. We observe that if the linear map $\mathrm{id}_M + B \circ T : M \rightarrow M$ is invertible, then $\tau_B ( \mathrm{Gr}(T) )$ is the graph of a linear map $T \circ (\mathrm{id}_M + B \circ T )^{-1} : M \rightarrow A$. In such a case, by Proposition \ref{dirac-char}, the linear map $T \circ (\mathrm{id}_M + B \circ T )^{-1} $ is a relative Rota-Baxter operator on the involutive algebra $A$ with respect to the involutive bimodule $M$. The relative Rota-Baxter operator $T \circ (\mathrm{id}_M + B \circ T )^{-1} $ is called the gauge transformation of $T$ associated with $B$.

\subsection{Maurer-Cartan characterization and cohomology}
In this subsection, we first recall from \cite{das-rota} that ordinary relative Rota-Baxter operators are Maurer-Cartan elements in a suitable graded Lie algebra $\mathfrak{g}$. Then we will show that involutive relative Rota-Baxter operators are Maurer-Cartan elements in a suitable graded Lie subalgebra of $\mathfrak{g}$.

Let $A$ be an ordinary associative algebra with product $\mu$ and $M$ be an $A$-bimodule with left and right $A$ actions $l, r$. Then the graded space $\oplus_{n \geq 0} \mathrm{Hom}(M^{\otimes n}, A)$ carries a graded Lie bracket defined by Voronov's derived bracket
\begin{align}\label{deri-br}
\llbracket P, Q \rrbracket := (-1)^m [ [ \mu + l + r , P ], Q ],
\end{align}
for $P \in \mathrm{Hom}(M^{\otimes m}, A), Q \in \mathrm{Hom}(M^{\otimes n}, A)$. Here $\mu + l + r$ can be considered as an element in $\mathrm{Hom}( (A \oplus M)^{\otimes 2}, A \oplus M)$. Similarly, $P$ can be considered as element in  $\mathrm{Hom}((A \oplus M)^{\otimes m}, A \oplus M)$ and same for $Q$. Finally, the bracket $[~,~]$ on the right hand side of (\ref{deri-br}) is the Gerstenhaber's bracket (\ref{gers-brkt}) on multilinear maps on the vector space $A \oplus M$. Explicitly, the bracket (\ref{deri-br}) is given by
\begin{align}\label{derived-bracket}
&\llbracket P, Q \rrbracket  (u_1, \ldots, u_{m+n}) \\
&= \sum_{i=1}^m (-1)^{(i-1) n} P (u_1, \ldots, u_{i-1}, Q (u_i, \ldots, u_{i+n-1}) u_{i+n}, \ldots, u_{m+n} ) \nonumber\\
~&- \sum_{i=1}^m (-1)^{in} P (u_1, \ldots, u_{i-1}, u_i Q (u_{i+1}, \ldots, u_{i+n}), u_{i+n+1}, \ldots, u_{m+n}) \nonumber\\
&- (-1)^{mn} \bigg\{ \sum_{i=1}^n (-1)^{(i-1)m}~ Q (u_1, \ldots, u_{i-1}, P (u_i, \ldots, u_{i+m-1}) u_{i+m}, \ldots, u_{m+n}) \nonumber\\
&- \sum_{i=1}^n (-1)^{im}~ Q (u_1, \ldots, u_{i-1}, u_i P (u_{i+1}, \ldots, u_{i+m}) , u_{i+m+1}, \ldots, u_{m+n} ) \bigg\} \nonumber\\
&+ (-1)^{mn} \big[ P(u_1, \ldots, u_m) Q (u_{m+1}, \ldots, u_{m+n} ) - (-1)^{mn}~ Q (u_1, \ldots, u_n) P ( u_{n+1}, \ldots, u_{m+n} ) \big], \nonumber
\end{align}
\begin{align*}
&\llbracket P, a \rrbracket (u_1, \ldots, u_m ) \\&= \sum_{i=1}^m P (u_1, \ldots, u_{i-1}, a u_i - u_i a , u_{i+1}, \ldots, u_{m} ) 
+ P (u_1, \ldots, u_m) a - a  P (u_1, \ldots, u_m), \nonumber \\
& \text{ and } \llbracket a, b \rrbracket = ab - ba,
\end{align*}
for $P \in \mathrm{Hom}(M^{\otimes m}, A), ~ Q \in \mathrm{Hom}(M^{\otimes n}, A),$ $a, b \in A$ and $u_1, \ldots, u_{m+n} \in M.$

It is easy from the above bracket that a linear map $T \in \mathrm{Hom}(M, A)$ is an ordinary relative Rota-Baxter operator on $A$ with respect to the $A$-bimodule $M$ if and only if $T$ is a Maurer-Cartan element in the above-graded Lie algebra. The cohomology induced from the Maurer-Cartan element $T$ is called the cohomology of the relative Rota-Baxter operator $T$, and they are denoted by $H^\bullet_T (M, A)$.

Next, let $A$ be an involutive associative algebra and $M$ be an involutive $A$-bimodule. Consider the graded space of involutive multilinear maps $\oplus_{n \geq 0} i \mathrm{Hom}(M^{\otimes n}, A)$, where $i \mathrm{Hom}(M^{\otimes 0}, A) = iA = \{ a \in A|a^* = -a \}$ and for $n \geq 1$,
\begin{align*}
i\mathrm{Hom}(M^{\otimes n}, A) = \{ f \in \mathrm{Hom}(M^{\otimes n}, A) | ~ f (u_1, \ldots, u_n )^* = (-1)^{\frac{(n-1)(n-2)}{2}} f (u_n^*, \ldots, u_1^*) \}.
\end{align*}

Since involutive multilinear maps are closed under the Gerstenhaber's bracket, it follows that the bracket (\ref{deri-br}) restricts to the graded subspace $\oplus_{n \geq 0} i \mathrm{Hom}(M^{\otimes n}, A)$ by the same formula as (\ref{derived-bracket}). It follows that a linear map $T : M \rightarrow A$ is a involutive relative Rota-Baxter operator if and only if $T \in i \mathrm{Hom}(M,A)$ is a Maurer-Cartan element in the graded Lie algebra $( \oplus_{n \geq 0} i \mathrm{Hom}(M^{\otimes n }, A), \llbracket ~, ~ \rrbracket )$.

Thus, an involutive relative Rota-Baxter operator $T$ induces a degree $1$ differential $d_T = \llbracket T, ~ \rrbracket$ on the graded space $\oplus_{n \geq 0} i \mathrm{Hom}(M^{\otimes n }, A)$. The corresponding cohomology groups are called the cohomology of the involutive relative Rota-Baxter operator $T$, and they are denoted by $iH^\bullet_T (M, A)$.

\section{Some properties of the cohomology}\label{sec-coho}
In this section, we first show that the cohomology of an involutive relative Rota-Baxter operator can be seen as the Hochschild cohomology of an involutive associative algebra. We also obtain a splitting theorem of the ordinary cohomology of a relative Rota-Baxter operator on an involutive associative algebra. Finally, we relate the cohomology of an involutive relative Rota-Baxter operator to the cohomology of the corresponding involutive dendriform algebra.

\subsection{Cohomology as involutive Hochschild cohomology}
Let $T: M \rightarrow A$ be a relative Rota-Baxter operator on an involutive associative algebra $A$ with respect to the involutive $A$-bimodule $M$. Then by Proposition \ref{inv-rota-inv-dend}, $M$ carries an involutive dendriform algebra structure. Hence $M$ has an involutive associative algebra structure with product
\begin{align*}
u \circledast v = u T(v) + T(u)v,~ \text{ for } u,v\in M.
\end{align*}
The following lemma is a generalization of \cite{uchino} in the involutive context.

\begin{lemma}
Let $T: M \rightarrow A$ be a relative Rota-Baxter operator on an involutive associative algebra $A$ with respect to the involutive $A$-bimodule $M$. Then the maps
\begin{align*}
l_T : M \otimes A \rightarrow A,~ (u,a) \mapsto T(u) a - T(ua),\\
r_T : A \otimes M \rightarrow A,~ (a,u) \mapsto a T(u) - T(au)
\end{align*}
defines an involutive $M$-bimodule structure on $A$.
\end{lemma}

\begin{proof}
In \cite{uchino} it has been proved that the maps $l_T$ and $r_T$ define an $M$-bimodule structure on $A$. Thus we need to verify the compatibility of involution and the maps $l_T, r_T$. We have
\begin{align*}
l_T (u, a)^* = a^* T(u)^* - T((ua)^*) = a^* T(u^*) - T(a^* u^*) = r_T ( a^*, u^*).
\end{align*}
Similarly, $r_T (a,u)^* = l_T (u^*, a^*)$. Hence the proof.
\end{proof}

It follows from the above lemma that we may consider the Hochschild cochain complex of the involutive associative algebra $M$ with coefficients in the involutive $M$-bimodule $A$. More precisely, we consider the cochain complex $\{ iC^\bullet_{\mathrm{Hoch}} (M,A), \delta^T_{\mathrm{Hoch}} \}$, where $iC^0_{\mathrm{Hoch}}(M,A) = \{ a \in A | a^* = -a \}$ and
\begin{align*}
iC^n_{\mathrm{Hoch}}(M,A) = \{ f : M^{\otimes n} \rightarrow A |~ f (u_1, \ldots, u_n )^* = (-1)^{\frac{(n-1)(n-2)}{2}} f (u_n^*, \ldots, u_1^*) \},~ \text{ for } n \geq 1
\end{align*}
and the differential $\delta^T_{\mathrm{Hoch}} : iC^n_{\mathrm{Hoch}}(M,A) \rightarrow iC^{n+1}_{\mathrm{Hoch}}(M,A)$ given by
\begin{align}\label{delta-ho}
( \delta^T_{\mathrm{Hoch}} f ) (u_1, \ldots, u_{n+1}) =~& l_T (u_1, f(u_2, \ldots, u_{n+1}) )
+ \sum_{i=1}^n (-1)^i f (u_1, \ldots, u_{i-1}, u_i \circledast u_{i+1}, \ldots, u_{n+1}) \\
~&+ (-1)^{n+1}~ r_T (f (u_1, \ldots, u_n), u_{n+1}). \nonumber
\end{align}
It has been shown in \cite{das-rota} that the coboundary operator $d_T$ induced from the Maurer-Cartan element $T$ and the coboundary operator (\ref{delta-ho}) are related by
\begin{align*}
d_T f = (-1)^n \delta^T_{\mathrm{Hoch}} f,~\text{ for }f \in iC^n_{\mathrm{Hoch}}(M,A) = i\mathrm{Hom}(M^{\otimes n}, A).
\end{align*}

Hence we get that the cohomology of the involutive relative Rota-Baxter operator $T$ is isomorphic to the Hochschild cohomology of the involutive associative algebra $M$ with coefficients in the involutive $M$-bimodule $A$.

\subsection{Splitting theorem}
In \cite{braun} Braun has shown that for involutive associative algebras, the ordinary Hochschild cohomology splits as a direct sum of involutive Hochschild cohomology and a skew-factor. This splitting theorem has been explicitly described in a recent paper by the present author \cite{das-saha} and further extended it to the dendriform context. Here we conclude a similar result for relative Rota-Baxter operators.

Let $T: M \rightarrow A$ be a relative Rota-Baxter operator on an involutive associative algebra $A$ with respect to an involutive $A$-bimodule $M$. 
For each $n \geq 0$, consider a linear map $S_n : \mathrm{Hom}( M^{\otimes n }, A) \rightarrow  \mathrm{Hom}( M^{\otimes n }, A)$ by
\begin{align*}
S_0 (a) = - a^* ~~~\text{ and } ~~~(S_n P) (a_1, \ldots, a_n ) = (-1)^{\frac{(n-1)(n-2)}{2}} P (a_n^*, \ldots, a_1^*)^*,~ \text{ for } n \geq 1.
\end{align*}
Then we have $(S_n)^2 = \mathrm{id}.$ Therefore, the map $S_n$ has eigenvalues $\pm 1$. Observe that the eigenspace corresponding to the eigenvalue $+1$ is precisely given by $i \mathrm{Hom} (M^{\otimes n},A)$. Denote the eigenspace corresponding to the eigenvalue $-1$ by $i_{-} \mathrm{Hom} (M^{\otimes n},A)$. Then we have
\begin{align}\label{split-iso}
\mathrm{Hom} (M^{\otimes n},A) \cong  i \mathrm{Hom} (M^{\otimes n},A) \oplus i_{-} \mathrm{Hom} (M^{\otimes n},A), ~ f \mapsto \bigg( \frac{f + S_n f}{2}, \frac{f-S_n f}{2} \bigg).
\end{align}
It is easy to verify that $\{ i_{-} \mathrm{Hom} (M^{\otimes \bullet},A) , d_T \}$ is a subcomplex of the complex $\{ \mathrm{Hom} (M^{\otimes \bullet},A), d_T \}$. We denote the corresponding cohomology groups by $i_{-}H^\bullet_T (M, A)$.
Note that the isomorphisms (\ref{split-iso}) preserve the corresponding differentials on both sides. Hence we get the following.
\begin{prop}
For an involutive relative Rota-Baxter operator $T$, the ordinary cohomology of $T$ splits as a direct sum $H^\bullet_T (M, A) \cong iH^\bullet_T (M, A) \oplus i_{-} H^\bullet_T (M, A).$
\end{prop}

\subsection{Relation with the cohomology of involutive dendriform algebras}
The cohomology of dendriform algebras was first defined by Loday \cite{loday} with trivial coefficients and the operadic approach was given in \cite{loday-val}. An explicit description of the cohomology was given in \cite{das-dend}. Here we require the cohomology of involutive dendriform algebras given in \cite{das-saha}.

Let $C_n$ be the set of first $n$ natural numbers. For convenience, we denote the elements of $C_n$ by $\{[1], [2], \ldots, [n] \}.$ It has been shown in \cite{das-dend} that for any vector space $D$, the collection of spaces
\begin{align*}
\mathcal{O}(n) = \mathrm{Hom}( \mathbb{K}[C_n] \otimes D^{\otimes n}, D), ~ \text{ for } n \geq 1
\end{align*}
forms a non-symmetric operad with partial compositions

\medskip

\noindent $(f \circ_i g)([r]; a_1, \ldots, a_{m+n-1}) =$
\begin{align*}
\begin{cases}
f ([r]; a_1, \ldots, a_{i-1}, g ([1]+ \cdots + [n]; a_i, \ldots, a_{i+n-1}), \ldots, a_{m+n-1}) ~& \text{ if } 1 \leq r \leq i-1 \\
f ([i]; a_1, \ldots, a_{i-1}, g ([r-i+1]; a_i, \ldots, a_{i+n-1}), \ldots, a_{m+n-1}) ~& \text{ if } i \leq r \leq i+ n-1\\
f ([r-n+1]; a_1, \ldots, a_{i-1}, g ([1]+ \cdots + [n]; a_i, \ldots, a_{i+n-1}), \ldots, a_{m+n-1}) ~& \text{ if } i+ n \leq r \leq m + n -1,
\end{cases}
\end{align*}
for $f \in \mathcal{O}(m), ~ g \in \mathcal{O}(n),~ 1 \leq i \leq m$ and $[r] \in C_{m+n-1}$. Therefore, there is a graded Lie bracket on the graded vector space $\mathcal{O}(\bullet +1 ) = \oplus_{n \geq 0} \mathcal{O}(n+1)$ given by
\begin{align*}
\llceil f, g \rrceil = \sum_{i=1}^{m+1} (-1)^{(i-1)n} f \circ_i g ~-~ (-1)^{mn} \sum_{i=1}^{n+1} (-1)^{(i-1)m} g \circ_i f,
\end{align*}
for $f \in \mathcal{O}(m+1)$ and $g \in \mathcal{O}(n+1)$. More generally, if $(D, \prec, \succ)$ is a dendriform algebra, then the element $\pi \in \mathcal{O}(2)$ defined by
\begin{align*}
\pi ([1]; a, b ) = a \prec b ~~~~ \text{ and } ~~~~ \pi ([2]; a, b ) = a \succ b
\end{align*}
satisfies $\llceil \pi, \pi \rrceil = 0$, i.e. $\pi$ defines a Maurer-Cartan element in the above graded Lie algebra. Hence $\pi$ induces a differential $\delta_\pi : \mathcal{O}(n) \rightarrow \mathcal{O}(n+1)$ given by $\delta_\pi (f) := (-1)^{n-1} \llceil \pi, f \rrceil$, for $f \in \mathcal{O}(n)$. 

Let $(D, \prec, \succ, *)$ be an involutive dendriform algebra. We define 
\begin{align*}
iC^n_{\mathrm{dend}} (D, D) = \{ f \in \mathcal{O}(n) |~ f ([r]; a_1, \ldots, a_n )^* = (-1)^{\frac{(n-1)(n-2)}{2}} f([n-r+1]; a_n^*, \ldots, a_1^*) \}, ~ \text{ for } n \geq 1.
\end{align*}
Then it has been shown in \cite{das-saha} that $\{ iC^\bullet_\mathrm{dend} (D, D) , \delta_\pi \}$ is a subcomplex of the cochain complex $\{ \mathcal{O}(\bullet), \delta_\pi \}$.
The cohomology groups of this subcomplex are called the cohomology of the involutive dendriform algebra $(D, \prec, \succ, *)$ and they are denoted by $iH^\bullet_\mathrm{dend}(D, D)$. \\

Let $T$ be a relative Rota-Baxter operator on an involutive associative algebra $A$ with respect to an involutive $A$-bimodule $M$. Consider the involutive dendriform algebra structure on $M$. We denote by $\pi_T \in iC^2_{\mathrm{dend}}(M,M)$ the corresponding Maurer-Cartan element. Define a collection of maps $\Theta_n : i\mathrm{Hom}(M^{\otimes n}, A) \rightarrow iC^{n+1}_{\mathrm{dend}} (M,M)$ by
\begin{align*}
\Theta_n (P ) ([r]; u_1, u_2, \ldots, u_{n+1}) = \begin{cases} (-1)^{n+1} ~u_1 P(u_2, \ldots u_{n+1})  ~~~& \text{if }~ r = 1 \\
0 ~~~& \text{if } ~2 \leq r \leq n \\
P(u_1, \ldots, u_n ) u_{n+1} ~~~& \text{if }~ r = n + 1. \end{cases}
\end{align*}
Note that $\Theta_n (P) \in iC^{n+1}_{\mathrm{dend}}(M,M)$ as
\begin{align*}
\Theta_n (P)([1]; u_1, \ldots, u_{n+1})^* = (-1)^{n+1} P(u_2, \ldots, u_{n+1})^* u_1^* =~& (-1)^{n+1} (-1)^{\frac{(n-1)(n-2)}{2}} ~P(u_{n+1}^*, \ldots, u_2^*) u_1^* \\
=~& (-1)^{\frac{n(n-1)}{2}} P(u_{n+1}^*, \ldots, u_2^*) u_1^* \\
=~& (-1)^{\frac{n(n-1)}{2}} \Theta_n (P)([n+1]; u_{n+1}^*, \ldots, u_1^*).
\end{align*}
For $2 \leq r \leq n$, we have $\Theta_n (P)([r]; u_1, \ldots, u_{n+1})^* = 0 = \Theta_n (P) ([n-r+2];  u_{n+1}^*, \ldots, u_1^*).$ 

With these notations, we have the following \cite[Lemma 3.4]{das-rota}.

\begin{lemma}
The collection $ \{ \Theta_n \}$ of maps preserve the corresponding graded Lie brackets, i.e.
\begin{align*}
\llceil \Theta_m (P), \Theta_n (Q) \rrceil = \Theta_{m+n} (\llbracket P, Q \rrbracket).
\end{align*}
\end{lemma}

Hence as a consequence, we get the following.

\begin{prop}
Let $T$ be a relative Rota-Baxter operator on an involutive associative algebra $A$ with respect to the involutive $A$-bimodule $M$. Then the collection $\{ \Theta_n \}$ of maps induces a morphism $\Theta_* : iH^\bullet_T (M, A) \rightarrow iH^{\bullet + 1}_{\mathrm{dend}} (M,M)$ from the cohomology of $T$ to the cohomology of the involutive dendriform algebra structure on $M$.
\end{prop}

\section{Deformations}\label{sec-def}
In this section, we study formal deformations of relative Rota-Baxter operators on involutive associative algebras from cohomological perspectives.

Let $A$ be an involutive associative algebra and $M$ be an involutive $A$-bimodule. Consider the space $A[[t]]$ of formal power series in $t$ with coefficients from $A$. The involution on $A$ induces an involution on $A[[t]]$ and the associative multiplication on $A$ induces an associative multiplication on $A[[t]]$ by $\mathbb{K}[[t]]$-bilinearity. With these structures, $A[[t]]$ is an involutive associative algebra. Moreover, the space $M[[t]]$ can be given the structure of an involutive $A[[t]]$-bimodule with the obvious left and right actions.

\begin{defn}
Let $T: M \rightarrow A$ be a relative Rota-Baxter operator on the involutive algebra $A$ with respect to the involutive $A$-bimodule $M$. A formal one-parameter deformation of $T$ consists of a formal sum
\begin{align*}
T_t = T_0 + tT_1 + t^2 T_2  + \cdots \in \mathrm{Hom}(M,A)[[t]]
\end{align*}
in which $T_0 = T$ such that as a $\mathbb{K}[[t]]$-linear map $T_t : M[[t]] \rightarrow A[[t]]$ is a relative Rota-Baxter operator on the involutive algebra $A[[t]]$ with respect to the involutive $A[[t]]$-bimodule $M[[t]]$.
\end{defn}

Thus, the followings are hold: $T_t (u^*) = T_t (u)^*$ and 
\begin{align*}
T_t (u) T_t (v) = T_t ( u T_t (v) + T_t (u) v ), ~\text{ for } u, v \in M.
\end{align*}
These conditions are equivalent to the followings: for each $k \geq 0$, we have $T_k (u^*) = T_k (u)^*$ and
\begin{align*}
\sum_{i+j=k} T_i (u) T_j (v) = T_i ( u T_j (v) + T_j (u) v ), ~\text{ for } u, v \in M.
\end{align*}
For $k = 1$, we get $T_1 (u^*) = T_1(u)^*$ and 
\begin{align*}
T(u) T_1 (v) + T_1(u) T(v) =  T ( u T_1 (v) + T_1 (u) v ) + T_1 ( u T (v) + T (u) v ).
\end{align*}
This says that $T_1 \in i \mathrm{Hom}(M,A)$ is a $1$-cocycle in the cohomology of the involutive relative Rota-Baxter operator $T$.

\begin{defn}
Two deformations $T_t = \sum_{i \geq 0} t^i T_i$ and $T_t' = \sum_{i \geq 0} t^i T'_i$ of an involutive relative Rota-Baxter operator $T$ are said to be equivalent if there is an element ${\bf a} \in A$ with ${\bf a}^* = - {\bf a}$ and linear maps $\phi_j \in i \mathrm{Hom}(A, A)$, $\psi_j \in i \mathrm{Hom}(M,M)$, for $j \geq 2$ such that
\begin{align*}
\big(   \phi_t = \mathrm{id}_A + t (\mathrm{ad}_{\bf a}^l - \mathrm{ad}^r_{\bf a}) + \sum_{j \geq 2} t^j \phi_j,~ 
\psi_t = \mathrm{id}_M + t (l_{\bf a} - r_{\bf a}) + \sum_{j \geq 2} t^j \psi_j  \big)
\end{align*}
defines a morphism of relative Rota-Baxter operators from $T_t$ to $T_t'$.
\end{defn}

Hence by Definition \ref{mor-defn}, the following conditions must hold: for all $a,b \in A$ and $u \in M$,
\begin{align*}
\phi_t (a) \phi_t (b) =~ \phi_t (ab),  \quad &T_t' \circ \psi_t (u) =~ \phi_t \circ T_t (u), \quad \psi_t (au) =~ \phi_t (a ) \psi_t (u) \quad \\
&\text{ and } ~~~ \psi_t (ua) =~ \psi_t (u) \phi_t (a).
\end{align*}
In the second equality, by equating coefficients of $t$ from both sides, we get
\begin{align*}
T_1(u) - T_1'(u) = T ({\bf a}u -u{\bf a}) - ( {\bf a} T(u) - T(u){\bf a}) = \delta^T_{\mathrm{Hoch}}({\bf a})(u).
\end{align*}

Summarizing the above discussions, we get the following.

\begin{thm}
Let $T_t = \sum_{i \geq 0} t^i T_i$ be a formal one-parameter deformation of an involutive relative Rota-Baxter operator $T$. Then the linear term $T_1$ is a $1$-cocycle in the cohomology of $T$ whose cohomology class depends only on the equivalence class of the deformation $T_t$.
\end{thm}

\subsection{Extensions of finite order deformations}
In this subsection, we consider extensions of a finite order deformation of an involutive relative Rota-Baxter operator $T$. Given a finite order deformation of $T$, we associate a second cohomology class in the cohomology of $T$. When the class is trivial, the deformation extends to next order.

Let $T: M \rightarrow A$ be a relative Rota-Baxter operator on an involutive associative algebra $A$ with respect to the involutive  $A$-bimodule $M$.

\begin{defn}
An order $N$ deformation of $T$ consists of a finite sum
$T_t =  \sum_{i=0}^N t^i T_i \in \mathrm{Hom}(M,A)[[t]]/ (t^{N+1})$ such that $T_0 = T$ and as a $\mathbb{K}[[t]]/ (t^{N+1})$-linear map $T_t : M[[t]]/ (t^{N+1}) \rightarrow A[[t]]/ (t^{N+1})$ is an involutive relative Rota-Baxter operator on $A[[t]]/ (t^{N+1})$ with respect to the involutive $A[[t]]/ (t^{N+1})$-bimodule $M[[t]]/ (t^{N+1})$.
\end{defn}

Therefore, we must have $T_k (u^*) = T_k(u)^*$ and
\begin{align*}
\sum_{i+j=k} T_i (u) T_j (v) = T_i ( u T_j (v) + T_j (u) v ), ~\text{ for } u, v \in M \text{ and } k =0,1, \ldots, N.
\end{align*}
The last condition is equivalent to the fact that
\begin{align*}
d_T (T_k) = - \frac{1}{2} \sum_{i+j = k , i, j \geq 1} \llbracket T_i, T_j \rrbracket, \text{ for } k =0,1, \ldots, N.
\end{align*}
\begin{defn}
A deformation $T_t = \sum_{i=0}^N t^i T_i$ of order $N$ is said to be extensible if there exists an element $T_{N+1} \in i \mathrm{Hom}(M,A)$ such that $\widetilde{T}_t = T_t + t^{N+1} T_{N+1}$ is a deformation of order $N+1$.
\end{defn}
In such a case, one more deformation equation needs to be satisfied, namely,
\begin{align}\label{n+1}
d_T (T_{N+1}) = - \frac{1}{2} \sum_{i+j = N+1 , i, j \geq 1} \llbracket T_i, T_j \rrbracket.
\end{align}
Note that the right hand side of (\ref{n+1}) depends only on $\{ T_1, \ldots, T_N \}$ and does'nt involve $T_{N+1}$. Hence it depends on the deformation $T_t$. This is called the obstruction to the extend the deformation $T_t$, denoted by $\mathrm{Ob}_{T_t}$.

\begin{prop}\label{obs-2}
$\mathrm{Ob}_{T_t}$ is a $2$-cocycle in the cohomology complex of $T$.
\end{prop}

\begin{proof}
See \cite[Proposition 4.17]{das-rota}.
\end{proof}

The above proposition shows that a finite order deformation $T_t$ gives rise to a second cohomology class $[\mathrm{Ob}_{T_t}] \in iH^2_T (M,A)$, called the obstruction class. 

Hence from (\ref{n+1}) and Proposition \ref{obs-2}, we get the following.
\begin{thm}
A finite order deformation $T_t$ of an involutive relative Rota-Baxter operator $T$ extends to a deformation of next order if and only if the corresponding obstruction class $[\mathrm{Ob}_{T_t}] \in iH^2_T (M,A)$ is trivial.
\end{thm}

\begin{corollary}
If $iH^2_T (M,A) = 0$ then every finite order deformation of $T$ extends to a deformation of next order.
\end{corollary}

\section{Fard-Guo functor for involutive algebras}\label{sec-fard-guo}
In \cite{fard-guo} Ebrahimi-Fard and Guo constructs the universal enveloping Rota-Baxter algebra of a dendriform algebra. Here we recall their construction and observe that it passes to the involutive case.

Let $B$ be a nonunitary associative algebra. Let $X$ be a basis for $B$, and let $X' = X \cup \{ \lfloor , \rfloor \}$. Here $\lfloor$ and $\rfloor$ are two symbols, called brackets. Let $M (X')$ be the free semigroup generated by $X'$.

There is a sequence $\{ \mathfrak{X}_n \} $ of subsets of $M(X')$ defined by the following recursive formula: $\mathfrak{X}_0 = X$ and for $n \geq 0$,
\begin{align*}
\mathfrak{X}_{n+1} = \bigg( \bigcup_{r \geq 1} (X \lfloor \mathfrak{X}_n \rfloor )^r \bigg) \bigcup \bigg( \bigcup_{r \geq 0} (X \lfloor \mathfrak{X}_n \rfloor )^r X  \bigg) \bigcup \bigg( \bigcup_{r \geq 1}  ( \lfloor \mathfrak{X}_n \rfloor X )^r  \bigg) \bigcup \bigg( \bigcup_{r \geq 0} ( \lfloor \mathfrak{X}_n \rfloor X)^r \lfloor \mathfrak{X}_n \rfloor \bigg) .
\end{align*}
Then $\mathfrak{X}_{n+1} \supset \mathfrak{X}_n$, for $n \geq 0$. Define $\mathfrak{X}_\infty = \cup_{n \geq 0} \mathfrak{X}_n = \lim_{\to} \mathfrak{X}_n $. The words of $\mathfrak{X}_\infty$ are called Rota-Baxter words. Every Rota-Baxter word ${\bf x} \neq {\bf 1}$ has a unique decomposition (called standard decomposition) ${ \bf x} = {\bf x}_1 \cdots {\bf x}_b$, where ${\bf x}_i, ~ 1 \leq i \leq b$, is alternatively in $X$ or in $\lfloor X_\infty \rfloor.$ The number $b$ is called the breadth of ${\bf x}$, denoted by $b ({\bf x})$. We define the head $h ({\bf x})$ of ${\bf x}$ to be $0$ (resp. $1$) if ${\bf x}_1$ is in $X$ (resp. in $\lfloor \mathfrak{X}_\infty \rfloor$). Similarly, the tail $t({\bf x})$ to be defined as $0$ (resp. $1$) if ${\bf x}_b$ is in $X$ (resp. in $\lfloor \mathfrak{X}_\infty \rfloor$). Finally, the depth of ${\bf x}$ is defined as $d ({\bf x}) = \mathrm{min}\{ n | {\bf x} \in \mathfrak{X}_n \}$.


Define $\varpi^{\mathrm{NC}, 0} (B) = \bigoplus_{ {\bf x} \in \mathfrak{X}_\infty} \mathbb{K} {\bf x}$. For ${\bf x}, {\bf x}' \in \mathfrak{X}_\infty$ with $t ( {\bf x}) \neq h ({\bf x}')$, we define a product ${\bf x} \diamond {\bf x}'$ by the concatenation. For ${\bf x}, {\bf x}' \in \mathfrak{X}_\infty$ with  $t ( {\bf x}) = h ({\bf x}')$, we define ${\bf x} \diamond {\bf x}' $ using the induction on $n = d({\bf x}) + d({\bf x}')$. 
If $n= 0$, then ${\bf x}, {\bf x}'$ is in $X$, hence in $B$, and the product ${\bf x} \diamond {\bf x}' := {\bf x} \cdot {\bf x}'$ (the product in $B$). Suppose the product is defined for $n =k \geq 0$ and we want to define for $n = k+1$. If $b ({\bf x}) = b({\bf x}') = 1$, then
\begin{align}\label{prod-fard}
{\bf x} \diamond {\bf x}' = \begin{cases} {\bf x} \cdot {\bf x}' ~~(\text{the product in } B)  & \text{ if } {\bf x}, {\bf x}' \in X \\
{\bf x}{\bf x}' \quad (\text{concatenation})  & \text{ if } {\bf x} \in X, {\bf x}' \in \lfloor \mathfrak{X}_\infty \rfloor \text{ or } {\bf x} \in \lfloor \mathfrak{X}_\infty \rfloor, {\bf x}' \in X \\
\lfloor \lfloor \overline{x} \rfloor \diamond \overline{x}' \rfloor + \lfloor \overline{x} \diamond \lfloor \overline{x}'\rfloor \rfloor  & \text{ if }  {\bf x} =\lfloor \overline{x} \rfloor , {\bf x}' = \lfloor \overline{x}' \rfloor \in \lfloor \mathfrak{X}_\infty \rfloor.
\end{cases}
\end{align}
Finally, if $b ({\bf x} ) > 1$ or $b({\bf x}') > 1$, take  ${\bf x} = {\bf x}_1 \cdots {\bf x}_b$ and ${\bf x}' = {\bf x}'_1 \cdots {\bf x}'_{b'}$ be standard decompositions of ${\bf x}$ and ${\bf x}'$. In this case, we define
\begin{align*}
{\bf x} \diamond {\bf x}' = {\bf x}_1 \cdots {\bf x}_{b-1} ( {\bf x}_b \diamond {\bf x}_1') {\bf x}'_2 \cdots {\bf x}'_{b'},
\end{align*}
where $ {\bf x}_b \diamond {\bf x}_1'$ is defined by (\ref{prod-fard}).
Then $(\varpi^{\mathrm{NC}, 0} (B), \diamond)$ is a nonunitary associative algebra and $R_B : \varpi^{\mathrm{NC}, 0} (B) \rightarrow \varpi^{\mathrm{NC}, 0} (B)$ defined by $R_B(x) = \lfloor x \rfloor$, for $x \in \mathfrak{X}_\infty$ is a Rota-Baxter operator on $(\varpi^{\mathrm{NC}, 0} (B), \diamond)$. We also consider the natural inclusion $j_X : X \rightarrow \mathfrak{X}_\infty \rightarrow \varpi^{\mathrm{NC}, 0} (B)$ which extends to an injective algebra map $j_B : B \rightarrow \varpi^{\mathrm{NC}, 0} (B)$. 

For any vector space $V$, consider the tensor algebra $T(V) = \oplus_{n \geq 1} V^{\otimes n}$. Then $(\varpi^{\mathrm{NC}, 0} (T(V)), \diamond, R_{T(V)})$ is a `free' nonunitary Rota-Baxter algebra over $V$ \cite{fard-guo}.
Let $(D, \prec, \succ)$ be a dendriform algebra. Consider the free nonunitary Rota-Baxter algebra $\varpi^{\mathrm{NC}, 0} (T(D))$ over the vector space $D$. Let $J_R$ be the Rota-Baxter ideal of $\varpi^{\mathrm{NC}, 0} (T(D))$ generated by the set
\begin{align*}
\{ x \prec y - x \lfloor y \rfloor, ~ x \succ y - \lfloor x \rfloor y ~|~ x, y \in D \}.
\end{align*}
Then the quotient Rota-Baxter algebra $\varpi^{\mathrm{NC}, 0} (T(D)) / J_R$ is the universal enveloping Rota-Baxter algebra of $D$.

Note that, if we start with a nonunitary involutive associative algebra $B$, then $(\varpi^{\mathrm{NC}, 0} (B), \diamond, R_B)$ can be given an involutive Rota-Baxter algebra with the involution given on basis elements by the involution on $B$ (when ${\bf x} \in X \subset B$),
\begin{align*}
\lfloor {\bf x} \rfloor^* = \lfloor {\bf x}^* \rfloor  ~~~ \text{ and } ~~~ ({\bf x}_1 \cdots {\bf x}_b)^* = {\bf x}_b^* \cdots {\bf x}_1^*.
\end{align*}


If $V$ is an involutive vector space, then $T(V)$ is an involutive algebra with involution $(v_1 \otimes \cdots \otimes v_n)^* = v_n^* \otimes \cdots \otimes v_1^*$. Hence $(\varpi^{\mathrm{NC}, 0} (T(V)) , \diamond, R_{T(V)}) $ is an involutive Rota-Baxter algebra. This is free in the following sense \cite{fard-guo}.
\begin{prop}
Let $V$ be an involutive vector space. Then for any nonunitary involutive Rota-Baxter algebra $A$ and a linear map $f: V \rightarrow A$ preserving involutions, there exists a unique nonunitary involutive Rota-Baxter algebra morphism $\widetilde{f} : \varpi^{\mathrm{NC}, 0} (T(V)) \rightarrow A$ such that $\widetilde{f} \circ (j_{T(V)} \circ i) = f$, where $i: V \rightarrow T(V)$ is the inclusion.
\end{prop}

Finally, for an involutive dendriform algebra $(D, \prec, \succ, *)$, the idear $J_R$ of the nonunitary involutive Rota-Baxter algebra $\varpi^{\mathrm{NC}, 0} (T(D))$ preserves under the involution as
\begin{align*}
( x \prec y - x \lfloor y \rfloor)^* = y^* \succ x^* - \lfloor y^* \rfloor x^* \in J_R ~~~ \text{ and } ~~~
( x \succ y - \lfloor x \rfloor y)^* = y^* \prec x^* - y^* \lfloor x^* \rfloor \in J_R.
\end{align*}
Hence we get the following.

\begin{prop}
If  $(D, \prec, \succ, *)$ is an involutive dendriform algebra then the universal enveloping Rota-Baxter algebra $\varpi^{\mathrm{NC}, 0} (T(D))/J_R$ is involutive.
\end{prop}

\noindent {\em Acknowledgements.} The research is supported by the fellowship of Indian Institute of Technology (IIT) Kanpur. The author thanks the Institute for support.

\end{document}